\documentclass{article}

\usepackage[english]{babel}

\usepackage[letterpaper,top=2cm,bottom=2cm,left=3cm,right=3cm,marginparwidth=1.75cm]{geometry}

\usepackage{amsmath}
\usepackage{amssymb}
\usepackage{amsthm}
\usepackage{graphicx}
\usepackage[colorlinks=true, allcolors=blue]{hyperref}

\usepackage{bm}
\usepackage{tikz}
\usepackage[makeroom]{cancel}
\usepackage{makecell}
\usepackage{pdflscape} 

\newcommand{\ssf}[1]{{\scriptstyle\mathsf{#1}}}
\newcommand{\onefour}{\smallint(\ssf{I},\ssf{IV})}
\newcommand{\twothree}{\smallint(\ssf{II},\ssf{III})}
\newcommand{\oneonethree}{\smallint(\ssf{I},\ssf{I},\ssf{III})}
\newcommand{\onetwotwo}{\smallint(\ssf{I},\ssf{II},\ssf{II})}
\newcommand{\onetwothree}{\smallint(\ssf{I},\ssf{II},\ssf{III})}
\newcommand{\twotwotwo}{\smallint(\ssf{II},\ssf{II},\ssf{II})}
\newcommand{\oneoneonetwo}{\smallint(\ssf{I},\ssf{I},\ssf{I},\ssf{II})}
\newcommand{\oneonetwotwo}{\smallint(\ssf{I},\ssf{I},\ssf{II},\ssf{II})}
\newcommand{\onetwoonetwo}{\smallint(\ssf{I},\ssf{II},\ssf{I},\ssf{II})}
\newcommand{\onetwotwotwo}{\smallint(\ssf{I},\ssf{II},\ssf{II},\ssf{II})}
\newcommand{\oneoneoneoneone}{\smallint(\ssf{I},\ssf{I},\ssf{I},\ssf{I},\ssf{I})}
\newcommand{\oneoneoneonetwo}{\smallint(\ssf{I},\ssf{I},\ssf{I},\ssf{I},\ssf{II})}

\theoremstyle{plain}
\newtheorem{theorem}{Theorem}
\newtheorem{lemma}[theorem]{Lemma}

\theoremstyle{definition}

\newtheorem{problem}[theorem]{Problem}

\title{The nerd snipers problem}
\author{Boris Alexeev \and Dustin G.\ Mixon\thanks{Department of Mathematics, The Ohio State University, Columbus, OH} \thanks{Translational Data Analytics Institute, The Ohio State University, Columbus, OH}}
\date{}

\begin{document}
\maketitle

\begin{abstract}
We correct errors that appear throughout ``The vicious neighbour problem'' by Tao and Wu.
\end{abstract}

\section{Introduction}

We seek to solve the following problem.

\begin{problem}
\label{prob.main}
Suppose $N$ nerds are distributed uniformly at random in a square region.
At 3:14pm, every nerd simultaneously snipes their nearest neighbor. 
What is the expected proportion $P_N$ of nerds who are left unscathed in the limit as $N\to\infty$?
\end{problem}

In 1967, Abilock~\cite{Abilock:67} posed a nonasymptotic version of this problem that did not specify the region.
After two decades, Dan Shine submitted (a less colorful version of) Problem~\ref{prob.main} to the magazine \textit{Omni}, which then offered a \$500 prize for its solution~\cite{Morris:86}.
The following year, \textit{Omni} announced~\cite{Morris:87} the prize-winning solution by Tao and Wu~\cite{TaoW:87}, which concluded that a proportion of $0.284051$ are left unscathed.
In 2015, Finch~\cite{Finch:15} attempted to replicate Tao and Wu's tour de force of integration, but after spotting several errors and discrepancies, he concluded~\cite{Finch:19} ``it seems doubtful'' they obtained the correct answer to Problem~\ref{prob.main}.
Sadly, the technical difficulty in Tao and Wu's approach precluded Finch from obtaining an improved estimate.
Around the same time, the problem managed to nerd snipe the authors (\`{a} la xkcd~\cite{Munroe:xkcd}) thanks to a question~\cite{Smith:mse} posed on the Mathematics Stack Exchange, where Winther~\cite{Winther:mse} reports a simulation result of $0.28418(1)$\footnote{\label{iso}Here and throughout, we adopt the ISO standard~\cite{ISO:08} for reporting uncertainty by providing digit(s) in parentheses that reflect the error of the final digit(s) of a point estimate. For example, the statement $P\approx 0.28418(1)$ means $P$ is approximately $0.28418$ with error bars $\pm 1\times 10^{-5}$, while $P\approx 0.28418556313(96)$ means $P$ is approximately $0.28418556313$ with error bars $\pm 96\times 10^{-11}$.}.
This paper represents what is hopefully a remission of our decade-long battle with this mathematical disease.

In the next section, we review how Tao and Wu~\cite{TaoW:87} reduced Problem~\ref{prob.main} to computing a certain integral.
Section~\ref{sec.reducing to multiple integrals} then rigorizes a change of variables that Tao and Wu implemented incorrectly.
The purpose of this change of variables is to make the integral explicit enough for numerical integration.
We report the results of numerical integration in Section~\ref{sec.numerical integration and verification}, and establish that
\[
P
:=\lim_{N\to\infty}P_N
\approx 0.28418556313(96).
\]
In addition, we discuss several approaches we used to verify (portions of) this result.

\section{The approach of Tao and Wu}

By taking a thermodynamic limit, one may instead consider the following model: place a nerd at the origin and distribute infinitely other nerds on the plane according to a Poisson point process with rate $\lambda=1$.
Then the limiting proportion $P$ is precisely the probability that the nerd at the origin is left unscathed.
For each $n\geq 1$, let $c_n$ denote the expected number of size-$n$ sets of points in the point process that snipe the origin.
Then an application of the inclusion--exclusion principle gives
\[
P
= 1 + \sum_{n=1}^\infty (-1)^n c_n.
\]
In order for $n$ generic points to shoot the origin, they must be closer to the origin than to each other; it follows that $c_n=0$ for every $n\geq6$ because of constraints on packing in the plane.
Meanwhile, since every nerd snipes exactly once, every nerd can expect to be sniped exactly once, meaning $c_1=1$.
Overall, 
\[
P
=c_2-c_3+c_4-c_5.
\]
To compute each of the remaining $c_n$s, we consider all configurations of size-$n$ sets of points that are closer to the origin than to each other.
Such a configuration $\{\bm{r}_1,\ldots,\bm{r}_n\}$ will snipe the origin precisely when no other points from the Poisson process are even closer to any of the $\bm{r}_i$s, that is, when the other points manage to avoid all $n$ of the disks $D_i$ centered at $\bm{r}_i$ of radius $|\bm{r}_i|$.
Since our process has rate $\lambda=1$, this occurs with probability 
\[
e^{-V_n(\bm{r}_1,\ldots,\bm{r}_n)},
\]
where $V_n(\bm{r}_1,\ldots,\bm{r}_n)$ denotes the area of the union $\bigcup_{i=1}^n D_i$.
It follows that
\begin{equation}
\label{eq.c_n integral}
c_n
=\frac{1}{n}\int d\bm{r}_1\cdots d\bm{r}_n\,e^{-V_n(\bm{r}_1,\ldots,\bm{r}_n)},
\end{equation}
where the integral is taken over all $n$-tuples $(\bm{r}_1,\ldots,\bm{r}_n)$ of nonzero points that are closer to the origin than to each other and are furthermore labeled in counterclockwise order about the origin.
(The factor of $n$ accounts for how a size-$n$ sniping set determines $n$ different $n$-tuples in counterclockwise order.)

In the next section, we rigorize a change of variables that Tao and Wu used to rewrite \eqref{eq.c_n integral} in a way that leverages the counterclockwise ordering of the $\bm{r}_i$s.
In doing so, we expose an important constraint that was missing from Tao and Wu's treatment.

\section{Reducing to multiple integrals}
\label{sec.reducing to multiple integrals}

The following lemma describes our parametrization of \eqref{eq.c_n integral}.
Reader beware: Our $t_i$ indices differ from the Tao--Wu indices in order to match the $\theta_i$ indices.
Note that our final constraint on $t_n$ was missing in Tao and Wu's original treatment, and this discrepancy will be relevant later.

\begin{lemma}
\label{lem.change of variables}
Let $\mathcal{X}\subseteq\mathbb{R}^{2n}$ denote the set of vectors $(\theta,\theta_1,\theta_2,\ldots,\theta_{n-1},r,t_1,t_2,\ldots,t_{n-1})$ that, together with $t_n:=\frac{1}{t_1\cdots t_{n-1}}$ and $\theta_n:=2\pi-(\theta_1+\cdots+\theta_{n-1})$, satisfy the constraints
\begin{align*}
\theta&\in[0,2\pi),\\
\theta_i&\in(\tfrac{\pi}{3},\tfrac{5\pi}{3}) \text{ for every } i\in[n],\\
r&\in(0,\infty),\\
t_i&\in(0,\infty) \text{ for every } i\in[n], \text{ and }\\
t_i&\in(2\cos\theta_i,\tfrac{1}{2\cos\theta_i}) \text{ for every } i\in[n] \text{ such that } \theta_i\in(\tfrac{\pi}{3},\tfrac{\pi}{2})\cup(\tfrac{3\pi}{2},\tfrac{5\pi}{3}).
\end{align*}
Let $\mathcal{Y}\subseteq(\mathbb{R}^2)^n$ denote the set of $n$-tuples of nonzero points in $\mathbb{R}^2$ that are closer to the origin than to each other and are labeled in counterclockwise order about the origin.
Then the map
\[
F\colon(\theta,\theta_1,\theta_2,\ldots,\theta_{n-1},r,t_1,t_2,\ldots,t_{n-1})
\mapsto
(\bm{r}_1,\ldots,\bm{r}_n)
\]
such that each $\bm{r}_i$ has argument $\theta+\sum_{k=1}^{i-1}\theta_k$ and magnitude $r\prod_{k=1}^{i-1}t_k$ is a bijection $F\colon\mathcal{X}\to\mathcal{Y}$.
\end{lemma}

\begin{proof}
It is straightforward to verify that $F$ is injective.
In what follows, we show that $F$ is surjective.

First, we establish $\operatorname{im}(F)\subseteq\mathcal{Y}$.
Since the arguments are strictly increasing in $i$ and reside in $[\theta,\theta+2\pi)$, it follows that the $\bm{r}_i$s are labeled in counterclockwise order about the origin.
Also, each $\bm{r}_i$ has positive magnitude and is thus nonzero.
Next, when distinct $\bm{r}_i$s have interior angle at least $\frac{\pi}{2}$, we have
\[
|\bm{r}_i-\bm{r}_j|^2
\geq|\bm{r}_i|^2+|\bm{r}_j|^2
>\max\big\{|\bm{r}_i|^2,|\bm{r}_j|^2\big\},
\]
meaning they are closer to the origin than to each other.
Since the angle between adjacent $\bm{r}_i$s is greater than $\frac{\pi}{3}$, it remains to consider $\bm{r}_i$ and $\bm{r}_{i+1}$ with $\theta_i\in(\tfrac{\pi}{3},\tfrac{\pi}{2})\cup(\tfrac{3\pi}{2},\tfrac{5\pi}{3})$.
In this case, taking $s_i:=\min\{t_i,t_i^{-1}\}>2\cos\theta_i$ gives
\begin{align*}
|\bm{r}_i-\bm{r}_{i+1}|^2
&=|\bm{r}_i|^2-2|\bm{r}_i||\bm{r}_{i+1}|\cos\theta_i+|\bm{r}_{i+1}|^2\\
&=\max\big\{|\bm{r}_i|^2,|\bm{r}_{i+1}|^2\big\}\cdot\big(s_i\cdot(s_i-2\cos\theta_i)+1 \big)
>\max\big\{|\bm{r}_i|^2,|\bm{r}_{i+1}|^2\big\},
\end{align*}
meaning $\bm{r}_i$ and $\bm{r}_{i+1}$ are closer to the origin than to each other.

To show $\mathcal{Y}\subseteq\operatorname{im}(F)$, fix $(\bm{r}_1,\ldots,\bm{r}_n)\in\mathcal{Y}$.
We let $\theta\in[0,2\pi)$ and $r>0$ denote the argument and magnitude of $\bm{r}_1$, respectively.
Next, for each $i\in[n-1]$, let $\theta_i\in[0,2\pi)$ denote the difference of arguments between $\bm{r}_{i+1}$ and $\bm{r}_i$, and let $t_i>0$ denote the corresponding quotient of magnitudes.
(Note that $\theta_n\in[0,2\pi)$ and $t_n>0$ are then the difference in arguments and quotient of magnitudes, respectively, between $\bm{r}_1$ and $\bm{r}_n$.)
We claim that $(\theta,\theta_1,\theta_2,\ldots,\theta_{n-1},r,t_1,t_2,\ldots,t_{n-1})$ resides in $\mathcal{X}$.
First, we observe that 
\[
\max\big\{|\bm{r}_i|^2,|\bm{r}_{i+1}|^2\big\}
<|\bm{r}_i-\bm{r}_{i+1}|^2
=|\bm{r}_i|^2-2|\bm{r}_i||\bm{r}_{i+1}|\cos\theta_{i}+|\bm{r}_{i+1}|^2,
\]
and rearranging gives
\[
2\cos\theta_i
<\frac{|\bm{r}_i|^2+|\bm{r}_{i+1}|^2-\max\big\{|\bm{r}_i|^2,|\bm{r}_{i+1}|^2\big\}}{|\bm{r}_i||\bm{r}_{i+1}|}
=\min\{t_i,t_i^{-1}\}.
\]
Considering $\min\{t_i,t_i^{-1}\}\leq1$, it follows that $\theta_i>\frac{\pi}{3}$, and taking any $j\neq i$ gives $\theta_i\leq2\pi-\theta_j<\frac{5\pi}{3}$.
Furthermore, if $\theta_i\in(\tfrac{\pi}{3},\tfrac{\pi}{2})\cup(\tfrac{3\pi}{2},\tfrac{5\pi}{3})$, then the above inequality also gives that $t_i\in(2\cos\theta_i,\frac{1}{2\cos\theta_i})$.
Thus, our point is indeed a member of $\mathcal{X}$.
Finally, it is straightforward to verify that $F$ maps this point back to $(\bm{r}_1,\ldots,\bm{r}_n)$, and so $(\bm{r}_1,\ldots,\bm{r}_n)\in\operatorname{im}(F)$.
\end{proof}

To apply the change of variables in Lemma~\ref{lem.change of variables} to the integral~\eqref{eq.c_n integral}, we note that $V_n$ is invariant to $\theta$ and scales quadratically with $r$.
As such, we may write
\[
V_n
=r^2 \, W_n(\theta_1,\ldots,\theta_{n-1},t_1,\ldots,t_{n-1}),
\]
where $W_n$ denotes the area of the appropriate union of disks when $\bm{r}_1$ is a unit vector.
Next, one may verify that the determinant of the Jacobian of the map $F$ in Lemma~\ref{lem.change of variables} has absolute value
\[
r^{2n-1} \, t_1^{2n-3} \, t_2^{2n-5} \, \cdots 
 \, t_{n-2}^3 \, t_{n-1}.
\]
At this point, we can integrate out both $\theta$ and $r$ in general:
\[
\int_0^{2\pi} d\theta \int_0^\infty r^{2n-1} \, dr \, e^{-r^2W_n}
=(n-1)!\,\pi\,W_n^{-n}.
\]
It remains to integrate out each $\theta_i$ and $t_i$.
As we will see, the nature of our integrand and integral bounds depends on which quadrants the $\theta_i$s reside in for $i\in[n]$.
For this reason, we partition the integral by restricting to all feasible tuples of quadrants.
For example, one part of $c_2$ is given by
\[
\onefour
=\frac{1}{2}
\int_{\frac{\pi}{3}}^{\frac{\pi}{2}} d\theta_1 
\int_{2\cos\theta_1}^{\frac{1}{2\cos\theta_1}} t_1 \, dt_1 
\, \pi\,W_2^{-2},
\]
which corresponds to restricting $\theta_1$ to the first quadrant (denoted with the Roman numeral $\ssf{I}$), thereby forcing $\theta_2=2\pi-\theta_1$ to reside in the fourth quadrant (Roman numeral $\ssf{IV}$).
After accounting for all possible combinations of quadrants, we obtain
\begin{equation}
\label{eq.decompose c_2}
c_2
=\onefour+\twothree+\smallint(\ssf{III},\ssf{II})+\smallint(\ssf{IV},\ssf{I})
=2\onefour+2\twothree,
\end{equation}
where the last step applies the fact that $\smallint(\cdot)$ is invariant under cyclic (in fact, dihedral) permutations of the quadrants.
We may similarly decompose the other $c_n$s:
\begin{align}
\label{eq.decompose c_3}
c_3
&=3\oneonethree
+3\onetwotwo
+6\onetwothree
+\twotwotwo,
\\
\label{eq.decompose c_4}
c_4
&=4\oneoneonetwo
+4\oneonetwotwo
+2\onetwoonetwo
+4\onetwotwotwo,
\\
\label{eq.decompose c_5}
c_5
&=\oneoneoneoneone
+5\oneoneoneonetwo.
\end{align}
For the reader comparing with Tao and Wu~\cite{TaoW:87}, we note that
\[
I(0,0)
=2\twothree,
\quad
I(1,0)
=\onefour,
\]
\[
I(0,0,0)
=\twotwotwo,
\quad
I(1,0,0)
=\onetwotwo
+2\onetwothree,
\quad
I(1,1,0)
=\oneonethree,
\]
\[
I(1,0,0,0)
=\onetwotwotwo,
\quad
I(1,1,0,0)
=\oneonetwotwo,
\quad
I(1,0,1,0)
=\onetwoonetwo,
\quad
I(1,1,1,0)
=\oneoneonetwo,
\]
\[
I(1,1,1,1,0)
=\oneoneoneonetwo,
\quad
I(1,1,1,1,1)
=\oneoneoneoneone.
\]

Next, we determine the bounds of integration for each $\smallint$ above.
This will be particularly nontrivial in the case of $\oneoneoneoneone$ thanks to the constraints $\sum_{i=1}^n\theta_i=2\pi$ and $\prod_{i=1}^n t_i=1$, but the following lemma will help.

\begin{lemma}
\label{lem.integration bounds}
Fix $a_1,\ldots,a_n,b_1,\ldots,b_n,c\in\mathbb{R}$ such that $a_k\leq b_k$ for every $k\in[n]$ and consider the set
\[
R
:=\Big\{ ~ x\in\mathbb{R}^n ~ : 
~ 
x_1+\cdots+x_n = c
~~\text{and}~~
x_k\in[a_k,b_k]~\text{for all}~k\in[n]
~
\Big\}.
\]
For each $k\in[n]$, given $x\in\mathbb{R}^n$, denote $x_{<k} := (x_i)_{i<k}$, and given $y\in\mathbb{R}^{k-1}$, put
\[
\ell_k(y)
:= c - \sum_{i<k} y_i - \sum_{i>k} b_i,
\qquad
u_k(y)
:= c - \sum_{i<k} y_i - \sum_{i>k} a_i,
\]
where empty sums are zero by convention.
Then for every $k\in[n]$ and $y\in\mathbb{R}^{k-1}$ for which there exists $z\in R$ such that $z_{<k}=y$, it holds that
\begin{equation}
\label{eq.coordinate projection}
\pi_k\{x\in R:x_{<k}=y\}
=\big[ \ell_k( y ), u_k( y ) \big] \cap [ a_k, b_k ],
\end{equation}
where $\pi_k$ denotes projection onto the $k$th coordinate.
This set is nonempty and equals
\[
\big[ \max\{\ell_k( y ),a_k\}, \min\{u_k( y ), b_k\} \big].
\]
Furthermore,
\begin{align*}
a_k
\leq c-\sum_{i<k}b_i-\sum_{i>k}b_i
\qquad
&\implies
\qquad
\max\{\ell_k( y ),a_k\}=\ell_k( y ),\\
a_k
\geq c-\sum_{i<k}a_i-\sum_{i>k}b_i
\qquad
&\implies
\qquad
\max\{\ell_k( y ),a_k\}=a_k,\\
b_k
\geq c-\sum_{i<k}a_i-\sum_{i>k}a_i
\qquad
&\implies
\qquad
\min\{u_k( y ),b_k\}=u_k( y ),\\
b_k
\leq c-\sum_{i<k}b_i-\sum_{i>k}a_i
\qquad
&\implies
\qquad
\min\{u_k( y ),b_k\}=b_k.
\end{align*}
\end{lemma}

\begin{proof}
Fix $k\in[n]$ and $y\in\mathbb{R}^{k-1}$, and consider the sets
\[
S
:=\Big\{ ~ x\in\mathbb{R}^n ~ : 
~ 
x_1+\cdots+x_n = c,
~~~
x_i=y_i~\text{for all}~i<k,
~~\text{and}~
x_i\in[a_i,b_i]~\text{for all}~i>k
~
\Big\},
\quad
T
:=[a_k,b_k].
\]
Then $\{x\in R: x_{<k}=y\}=S\cap \pi_k^{-1}(T)$, and so
\[
\pi_k\{x\in R: x_{<k}=y\}
=\pi_k\big(S\cap \pi_k^{-1}(T)\big)
=\pi_k(S)\cap T.
\]
In addition, 
\[
\pi_k(S)
=\bigg\{ ~c-\sum_{i<k}y_i -\sum_{i>k}x_i~ : 
~
x_i\in[a_i,b_i]~\text{for all}~i>k
~
\bigg\}
=\big[ \ell_k( y ), u_k( y ) \big],
\]
and so \eqref{eq.coordinate projection} follows.
Next, the existence of $z\in R$ such that $z_{<k}=y$ implies that the left-hand set in \eqref{eq.coordinate projection} is nonempty.
In general, a nonempty intersection of intervals $[a,b]$ and $[c,d]$ is given by $[\max\{a,c\},\min\{b,d\}]$.
Finally, the first of the ``furthermore'' statements is established by
\[
a_k
\leq c-\sum_{i<k}b_i-\sum_{i>k}b_i
\leq c-\sum_{i<k}y_i-\sum_{i>k}b_i
= \ell_k(y),
\]
and the other statements have similar proofs.
\end{proof}

Applying Lemma~\ref{lem.integration bounds} to both $\{\theta_i\}_{i=1}^n$ and $\{\log t_i\}_{i=1}^n$, and using the shorthand notation $c_i:=2\cos\theta_i$, we obtain the following integration bounds:
\begingroup
\allowdisplaybreaks
\begin{align*}
\onefour
&=\frac{1}{2}
\int_{\frac{\pi}{3}}^{\frac{\pi}{2}} d\theta_1 
\int_{c_1}^{\frac{1}{c_1}} t_1 \, dt_1 
\, \pi\,W_2^{-2},\\
\twothree
&=\frac{1}{2}
\int_{\frac{\pi}{2}}^{\pi} d\theta_1 
\int_{0}^{\infty} t_1 \, dt_1 
\, \pi\,W_2^{-2},\\
\oneonethree
&=\frac{1}{3}
\int_{\frac{\pi}{3}}^{\frac{\pi}{2}} d\theta_1
\int_{\frac{\pi}{3}}^{\frac{\pi}{2}} d\theta_2 
\int_{c_1}^{\frac{1}{c_1}} t_1^3 \, dt_1
\int_{c_2}^{\frac{1}{c_2}} t_2 \, dt_2 
\, 2\pi\,W_3^{-3},\\
\onetwotwo
&=\frac{1}{3}
\int_{\frac{\pi}{3}}^{\frac{\pi}{2}} d\theta_1 
\int_{\pi-\theta_1}^{\pi} d\theta_2 
\int_{c_1}^{\frac{1}{c_1}} t_1^3 \, dt_1 
\int_{0}^{\infty} t_2 \, dt_2 
\, 2\pi \, W_3^{-3},
\\
\onetwothree
&=\frac{1}{3}
\int_{\frac{\pi}{3}}^{\frac{\pi}{2}} d\theta_1 
\int_{\frac{\pi}{2}}^{\pi-\theta_1} d\theta_2
\int_{c_1}^{\frac{1}{c_1}} t_1^3 \, dt_1 
\int_{0}^{\infty} t_2 \, dt_2 
\, 2\pi \, W_3^{-3},
\\
\twotwotwo
&=\frac{1}{3}
\int_{\frac{\pi}{2}}^{\pi} d\theta_1 
\int_{\frac{\pi}{2}}^{\frac{3\pi}{2}-\theta_1} d\theta_2 
\int_{0}^{\infty} t_1^3 \, dt_1 
\int_{0}^{\infty} t_2 \, dt_2 
\, 2\pi \, W_3^{-3},
\\
\oneoneonetwo
&=\frac{1}{4}
\int_{\frac{\pi}{3}}^{\frac{\pi}{2}} d\theta_1 
\int_{\frac{\pi}{3}}^{\frac{\pi}{2}} d\theta_2 
\int_{\frac{\pi}{3}}^{\frac{\pi}{2}} d\theta_3 
\int_{c_1}^{\frac{1}{c_1}} t_1^5 \, dt_1 
\int_{c_2}^{\frac{1}{c_2}} t_2^3 \, dt_2 
\int_{c_3}^{\frac{1}{c_3}} t_3 \, dt_3 
\, 6\pi \, W_4^{-4},
\\
\oneonetwotwo
&=\frac{1}{4}
\int_{\frac{\pi}{3}}^{\frac{\pi}{2}} d\theta_1 
\int_{\frac{\pi}{3}}^{\frac{\pi}{2}} d\theta_2 
\int_{\frac{\pi}{2}}^{\frac{3\pi}{2}-\theta_1-\theta_2} d\theta_3 
\int_{c_1}^{\frac{1}{c_1}} t_1^5 \, dt_1 
\int_{c_2}^{\frac{1}{c_2}} t_2^3 \, dt_2 
\int_{0}^{\infty} t_3 \, dt_3 
\, 6\pi \, W_4^{-4},
\\
\onetwoonetwo
&=\frac{1}{4}
\int_{\frac{\pi}{3}}^{\frac{\pi}{2}} d\theta_1 
\int_{\frac{\pi}{2}}^{\frac{7\pi}{6}-\theta_1} d\theta_2 
\int_{\frac{\pi}{3}}^{\min\{\frac{\pi}{2},\frac{3\pi}{2}-\theta_1-\theta_2\}} d\theta_3
\int_{c_1}^{\frac{1}{c_1}} t_1^5 \, dt_1 
\int_{0}^{\infty} t_2^3 \, dt_2 
\int_{c_3}^{\frac{1}{c_3}} t_3 \, dt_3 
\, 6\pi \, W_4^{-4},
\\
\onetwotwotwo
&=\frac{1}{4}
\int_{\frac{\pi}{3}}^{\frac{\pi}{2}} d\theta_1 
\int_{\frac{\pi}{2}}^{\pi-\theta_1} d\theta_2 
\int_{\frac{\pi}{2}}^{\frac{3\pi}{2}-\theta_1-\theta_2} d\theta_3 
\int_{c_1}^{\frac{1}{c_1}} t_1^5 \, dt_1 
\int_{0}^{\infty} t_2^3 \, dt_2 
\int_{0}^{\infty} t_3 \, dt_3 
\, 6\pi \, W_4^{-4},
\\
\oneoneoneoneone
&=\frac{1}{5}
\int_{\frac{\pi}{3}}^{\frac{\pi}{2}} d\theta_1 
\int_{\frac{\pi}{3}}^{\frac{\pi}{2}} d\theta_2 
\int_{\frac{\pi}{3}}^{\min\{\frac{\pi}{2},\frac{4\pi}{3}-\theta_1-\theta_2\}} d\theta_3 
\int_{\max\{\frac{\pi}{3},\frac{3\pi}{2}-\theta_1-\theta_2-\theta_3\}}^{\min\{\frac{\pi}{2},\frac{5\pi}{3}-\theta_1-\theta_2-\theta_3\}} d\theta_4 \\
&\qquad\qquad
\int_{\max\{c_1,c_2c_3c_4c_5\}}^{\min\{\frac{1}{c_1},\frac{1}{c_2c_3c_4c_5}\}} t_1^7 \, dt_1 
\int_{\max\{c_2,\frac{c_3c_4c_5}{t_1}\}}^{\min\{\frac{1}{c_2},\frac{1}{t_1c_3c_4c_5}\}} t_2^5 \, dt_2 \\
&\qquad\qquad
\int_{\max\{c_3,\frac{c_4c_5}{t_1t_2}\}}^{\min\{\frac{1}{c_3},\frac{1}{t_1t_2c_4c_5}\}} t_3^3 \, dt_3 
\int_{\max\{c_4,\frac{c_5}{t_1t_2t_3}\}}^{\min\{\frac{1}{c_4},\frac{1}{t_1t_2t_3c_5}\}} t_4 \, dt_4 
\, 24\pi \, W_5^{-5},
\\
\oneoneoneonetwo
&=\frac{1}{5}
\int_{\frac{\pi}{3}}^{\frac{\pi}{2}} d\theta_1 
\int_{\frac{\pi}{3}}^{\frac{5\pi}{6}-\theta_1} d\theta_2 
\int_{\frac{\pi}{3}}^{\frac{7\pi}{6}-\theta_1-\theta_2} d\theta_3 
\int_{\frac{\pi}{3}}^{\frac{3\pi}{2}-\theta_1-\theta_2-\theta_3} d\theta_4 \\
&\qquad\qquad
\int_{c_1}^{\frac{1}{c_1}} t_1^7 \, dt_1 
\int_{c_2}^{\frac{1}{c_2}} t_2^5 \, dt_2 
\int_{c_3}^{\frac{1}{c_3}} t_3^3 \, dt_3 
\int_{c_4}^{\frac{1}{c_4}} t_4 \, dt_4 
\, 24\pi \, W_5^{-5}.
\end{align*}
\endgroup

Note that the above expression for $\oneoneoneoneone$ is much more complicated than the expression for $I(1,1,1,1,1)$ supplied by Tao and Wu~\cite{TaoW:87}.
The reason for this discrepancy stems from a failure to properly account for the constraint $\prod_{i=1}^n t_i=1$, which we avoided by applying Lemma~\ref{lem.integration bounds}.
In fact, we claim that the Tao--Wu region of integration is neither a subset nor a superset of the desired region.
To see one direction of this claim, consider a small neighborhood of the point where $\theta_1=\theta_2=\theta_3=\theta_4=\frac{2\pi}{5}$ and $t_2=t_3=t_4=t_5=\frac{3}{2}$ (in Tao--Wu indexing).
While the Tao--Wu region of integration contains this neighborhood, every configuration in this neighborhood has the property that the fifth point fails to snipe the origin.
For the other direction, consider a small neighborhood of the point where $\theta_1=\theta_2=\theta_3=\theta_4=\frac{11\pi}{28}$, $t_2=t_3=t_4=\frac{3}{2}$, and $t_5=\frac{1}{2}$ (again, in Tao--Wu indexing).
While this avoids the Tao--Wu region of integration, every configuration in this neighborhood has the property that all five points snipe the origin.
See Figure~\ref{fig.certificate configurations} for an illustration of these configurations.

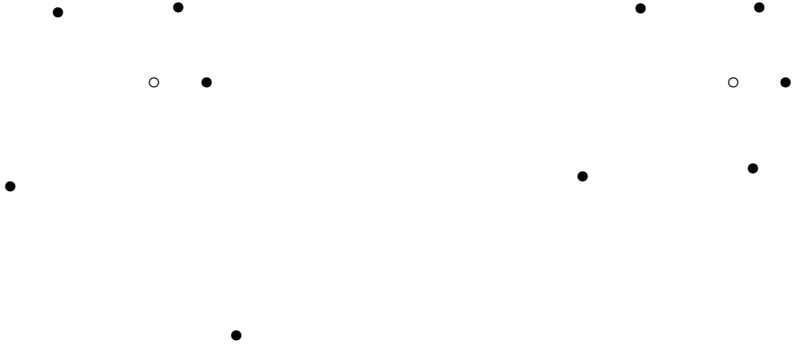
\begin{figure}
\begin{center}
\begin{tikzpicture}[scale=0.7]
\draw (0,0) node {$\circ$};
\draw (1,0) node {$\bullet$};
\draw ({3/2*cos(deg(2*pi/5))},{3/2*sin(deg(2*pi/5))}) node {$\bullet$};
\draw ({(3/2)^2*cos(deg(2*2*pi/5))},{(3/2)^2*sin(deg(2*2*pi/5))}) node {$\bullet$};
\draw ({(3/2)^3*cos(deg(3*2*pi/5))},{(3/2)^3*sin(deg(3*2*pi/5))}) node {$\bullet$};
\draw ({(3/2)^4*cos(deg(4*2*pi/5))},{(3/2)^4*sin(deg(4*2*pi/5))}) node {$\bullet$};
\draw (11+0,0) node {$\circ$};
\draw (11+1,0) node {$\bullet$};
\draw ({11+3/2*cos(deg(11*pi/28))},{3/2*sin(deg(11*pi/28))}) node {$\bullet$};
\draw ({11+(3/2)^2*cos(deg(2*11*pi/28))},{(3/2)^2*sin(deg(2*11*pi/28))}) node {$\bullet$};
\draw ({11+(3/2)^3*cos(deg(3*11*pi/28))},{(3/2)^3*sin(deg(3*11*pi/28))}) node {$\bullet$};
\draw ({11+1/2*(3/2)^3*cos(deg(4*11*pi/28))},{1/2*(3/2)^3*sin(deg(4*11*pi/28))}) node {$\bullet$};
\end{tikzpicture}
\end{center}
\caption{\label{fig.certificate configurations}
Two configurations that violate the Tao--Wu region of integration for $I(1,1,1,1,1)$.
In each illustration above, the white dot denotes the origin, while the black dots denote nearby snipers.
The left-hand configuration resides in the Tao--Wu region of integration, but the bottom-right point fails to snipe the origin.
Meanwhile, the right-hand configuration does not reside in the Tao--Wu region of integration even though every point snipes the origin.
We avoided these errors by applying Lemma~\ref{lem.integration bounds} when deriving our expression for $\oneoneoneoneone$.}
\end{figure}

We conclude this section by deriving a closed-form expression for each $W_n$.
Here, we follow the approach of Tao and Wu~\cite{TaoW:87}, but we avoid certain errors that appear in their original treatment by carefully stating $W_n$ in all cases.
For $n=2$, suppose $\bm{r}_1$ has unit norm, $\bm{r}_2$ has norm $t$, and their interior angle is $\theta\in(0,\pi)$.
Consider the following triangle:
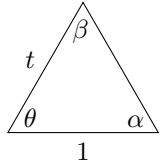
\[
\begin{tikzpicture}
\draw (0,0) -- (2,0) -- (1,1.732) -- cycle;
\draw (0.3,0.2) node {$\theta$};
\draw (1.7,0.15) node {$\alpha$};
\draw (0.97,1.35) node {$\beta$};
\draw (1,-0.25) node {$1$};
\draw (0.5-0.2,0.866+0.1) node {$t$};
%\draw (1.5+0.2,0.866+0.1) node {$s$};
\end{tikzpicture}
\]
Then the union $D_1\cup D_2$ can be partitioned into a section of $D_1$ with angle $2\pi-2\alpha$, a section of $D_2$ with angle $2\pi-2\beta$, a triangle in $D_1$ with side lengths $(1,1,2\sin\alpha)$, and a triangle in $D_2$ with side lengths $(t,t,2t\sin\beta)$.
Returning to our indexed integration variables, the total area is then given by
\[
W_2(\theta_1,t_1)
=(\pi-\alpha_1)+(\pi-\beta_1)t_1^2+\cos\alpha_1\sin\alpha_1+t_1^2\cos\beta_1\sin\beta_1,
\]
where here and throughout, we take
\[
\alpha_i
:=\arcsin\Big(\tfrac{t_i\sin\theta'_i}{\sqrt{1+t_i^2-2t_i\cos\theta'_i}}\Big),
\qquad
\beta_i
:=\arcsin\Big(\tfrac{\sin\theta'_i}{\sqrt{1+t_i^2-2t_i\cos\theta'_i}}\Big),
\qquad
\theta'_i
:=\min\{\theta_i,2\pi-\theta_i\}.
\]
(While we correctly distinguish between $\theta_i$ and $\theta'_i$ in the above, we note that for all of the integrals we compute, the bounds of integration ensure that $\theta'_i=\theta_i$ whenever $\alpha_i$ or $\beta_i$ appears in the integrand.) 
For $n=3$, if $\theta_1,\theta_2,\theta_3\in(0,\pi)$, then the union $D_1\cup D_2\cup D_3$ can be similarly partitioned into three sections and six triangles to obtain
\begin{align}
\nonumber
W_3(\theta_1,\theta_2,t_1,t_2)
&=(\pi-\alpha_1-\beta_3)+(\pi-\alpha_2-\beta_1)t_1^2+(\pi-\alpha_3-\beta_2)(t_1t_2)^2\\
\nonumber
&\qquad
+\cos\alpha_1\sin\alpha_1+t_1^2\cos\beta_1\sin\beta_1+t_1^2\cos\alpha_2\sin\alpha_2\\
\label{eq.n=3 with all angles interior}
&\qquad
+(t_1t_2)^2\cos\beta_2\sin\beta_2+(t_1t_2)^2\cos\alpha_3\sin\alpha_3+\cos\beta_3\sin\beta_3.
\end{align}
Considering \eqref{eq.decompose c_3}, the only other relevant possibility for $n=3$ is when $\theta_1,\theta_2\in(0,\pi)$ and $\theta_3\in(\pi,2\pi)$, in which case $D_1\cup D_2\cup D_3$ partitions into three sections and four triangles so that
\begin{align*}
W_3(\theta_1,\theta_2,t_1,t_2)
&=(\pi-\alpha_1)+(\pi-\alpha_2-\beta_1)t_1^2+(\pi-\beta_2)(t_1t_2)^2\\
&\qquad
+\cos\alpha_1\sin\alpha_1+t_1^2\cos\beta_1\sin\beta_1+t_1^2\cos\alpha_2\sin\alpha_2+(t_1t_2)^2\cos\beta_2\sin\beta_2.
\end{align*}
For $n\in\{4,5\}$, equations \eqref{eq.decompose c_4} and \eqref{eq.decompose c_5} allow us to assume $\theta_i\in(0,\pi)$ for all $i$, in which case the union $\bigcup_{i=1}^n D_i$ consists of $n$ sections and $2n$ triangles whose areas satisfy a generalization of \eqref{eq.n=3 with all angles interior}:
\[
W_n
=\sum_{i=1}^n \Big(\pi-\alpha_i-\beta_{i-1}+\cos\alpha_i\sin\alpha_i+\cos\beta_{i-1}\sin\beta_{i-1}\Big)\bigg(\prod_{k=1}^{i-1}t_k\bigg)^2,
\]
where the index of $\beta_{i-1}$ should be interpreted modulo $n$.

\section{Numerical integration and verification}
\label{sec.numerical integration and verification}

In this section, we estimate $P$ by numerically integrating the integrals we set up in the previous section, and we verify our computations by various means.
See Tables~\ref{table.p's} through~\ref{table.small ints} for a summary.

\begin{table}
\caption{\label{table.p's} Various estimates of $P$}
\begin{center}
\begin{tabular}{ll}
source & $P$ \\ \hline
Tao--Wu~\cite{TaoW:87}
& 0.284051
\\
Winther~\cite{Winther:mse}
& 0.28418(1)
\\
Monte Carlo integration
& 0.2841817(62)
\\
Monte Carlo simulation
& 0.28418587(20)
\\
numerical integration
& 0.28418556313(96)
\end{tabular}
\end{center}
\end{table}

\subsection{Numerical integration}

Our general approach was to use Mathematica's built-in \texttt{NIntegrate} function with precision goal set to $20$, and we increased the accuracy goal until the resulting runtime became a limiting factor.  (The high precision goal implies that the accuracy goal will be the constraining factor.)
To avoid long runtimes, we found that three tricks led to computational speedups.
First, it helped to split up our integrals at $t_i=1$. This splitting alone is a built-in feature of \texttt{NIntegrate}; for example, one can implement numerical integration over $t_1$ from $\frac{1}{2}$ to $2$ with a split at $1$ by replacing \verb|{t1,1/2,2}| with \verb|{t1,1/2,1,2}|.
However, we split the integrals into separate sub-integrals because after splitting at $1$, half of the sub-integral $t_i$~bounds end up over $(1,\frac{1}{c_i})$ or $(1,\infty)$, and we manually changed variables $u_i:=\frac{1}{t_i}$ in these cases to achieve a shorter interval of integration.
Finally, we observed that runtime suffers whenever one of the bounds of integration includes a~$\max$ or $\min$.
To mitigate this, we split up some of these integrals with the help of \textit{cylindrical algebraic decomposition}~\cite{Collins:75}, as implemented in Mathematica.
For example, for $\onetwoonetwo$, we decomposed
\begin{align*}
\int_{\frac{\pi}{3}}^{\frac{\pi}{2}} d\theta_1 
~\int_{\frac{\pi}{2}}^{\frac{7\pi}{6}-\theta_1} d\theta_2 
~\int_{\frac{\pi}{3}}^{\min\{\frac{\pi}{2},\frac{3\pi}{2}-\theta_1-\theta_2\}} d\theta_3
\quad
&=\quad
\int_{\frac{\pi}{3}}^{\frac{\pi}{2}} d\theta_1 
~\int_{\frac{\pi}{2}}^{\pi-\theta_1} d\theta_2 
~\int_{\frac{\pi}{3}}^{\frac{\pi}{2}} d\theta_3\\
&\quad
+
\int_{\frac{\pi}{3}}^{\frac{\pi}{2}} d\theta_1 
\int_{\pi-\theta_1}^{\frac{7\pi}{6}-\theta_1} d\theta_2 
\int_{\frac{\pi}{3}}^{\frac{3\pi}{2}-\theta_1-\theta_2} d\theta_3,
\end{align*}
which appeared to improve computational efficiency; it was not as useful to similarly decompose the bounds on $\theta$ in $\oneoneoneoneone$ into three parts.
Whenever we split up an integral, we estimated the worst-case total error by the triangle inequality; that is, the error bounds we report in the tables are the sum of the accuracy goals for all of the corresponding sub-integrals.

Nota bene, \texttt{NIntegrate} is not perfect, and in particular, we highlight the following disclaimer from Mathematica's documentation:
\begin{quote}
You should realize that with sufficiently pathological functions, the algorithms used by \texttt{NIntegrate} can give wrong answers. 
In most cases, you can test the answer by looking at its sensitivity to changes in the setting of options for \texttt{NIntegrate}.
\end{quote}
We took the above suggestion to heart by comparing results from a wide variety of options, but even so, we cannot be 100 percent confident that the true value of each $c_n$ and $P$ resides within the error bars we report.

\subsection{Verification}

Since the integral bounds for each $\smallint$ are so complicated, it is important to test their correctness.
To this end, we systematically searched for counterexamples like those in Figure~\ref{fig.certificate configurations}.
We randomly sampled points within each of our integral bounds and verified that they determine an appropriate geometry of points that snipe the origin, and conversely, we randomly sampled points that snipe the origin and verified that they satisfy the appropriate integral bounds.
We did not find any counterexamples to our stated bounds of integration (but the random searches quickly found counterexamples in both directions for Tao and Wu's bounds for $\oneoneoneoneone$).
Next, we note that some of the integral bounds in Tao and Wu degenerate.
For example, $I(1,1,1,1,0)$ integrates $\theta_4$ from $\min\{\frac{\pi}{3},\frac{3\pi}{2}-\theta_1-\theta_2-\theta_3\}$ to $\frac{3\pi}{2}-\theta_1-\theta_2-\theta_3$, which is an interval of length zero whenever $\theta_1+\theta_2+\theta_3\geq\frac{7\pi}{6}$.
We experienced issues with numerically integrating over bounds that degenerate in this way, and so we verified that our integral bounds never degenerate. 
In particular, we used cylindrical algebraic decomposition to check this algebraically, and then we further tested this by randomly sampling points within each of the integral bounds and confirming that downstream intervals always had positive length.

Next, we acknowledge that the piecewise-defined expression for $W_n$ is a bit tricky, and so we also wanted to test its correctness.
Thankfully, Mathematica has built-in functions that allow one to compute the area of a union of disks.
For example,
\begin{quote}
\begin{verbatim}
Area[ RegionUnion[ Disk[{1,0},1], Disk[{0,2},2] ] ]
\end{verbatim}
\end{quote}
returns the area of a union of two disks that matches our expression for $W_2(\frac{\pi}{2},2)$.
We randomly selected points that satisfy each of our integration bounds, and at each point, we verified that $W_n$ matches the area given by Mathematica's built-in functions to within a reasonable computational error.

Finally, it turns out that $c_2$ is so easy to compute that we were able to use two different methods to verify its computation.
First, following an idea of Finch~\cite{Finch:15}, we computed $c_2$ using
\begin{quote}
\begin{verbatim}
2 * Pi/2 * NIntegrate[
  Boole[ 1^2+0^2 <= (1-x2)^2+(0-y2)^2 && x2^2+y2^2 <= (1-x2)^2+(0-y2)^2 ] * 
  Area[ RegionUnion[ Disk[{1,0},1], Disk[{x2,y2},Sqrt[x2^2+y2^2]] ] ]^(-2), 
  {x2,-Infinity,Infinity}, {y2,0,Infinity},
  AccuracyGoal->12, PrecisionGoal->12, WorkingPrecision->20]
\end{verbatim}
\end{quote}
Second, Henze~\cite{Henze:87} derived a different integral expression for $c_2$ in a more general setting (i.e., in which snipers are randomly distributed in $\mathbb{R}^d$), and we numerically integrated this expression with $d=2$.
Both of these independent computations matched our estimate of $c_2$.
In effect, this verifies various aspects of our integrals (such as the Jacobian) that we did not verify by other means.

\subsection{Comparison}

As one can see from the tables, the values we compute for the various integrals differ from those provided by Tao--Wu~\cite{TaoW:87}.
Generally, Tao--Wu provide an accuracy of $7$ to $9$ digits after the decimal place (including zeros); despite the apparent precision, we find that the number of correct leading nonzero digits ranges from~$0$ (in the case of~$c_5$) to~$4$ (in the case of~$I(1,0)$).

Tao--Wu also made some arithmetic errors, as can be seen in these equations from their text:
\begin{subequations}
\begin{align}
c_3
&= I(0,0,0) + 3I(1,0,0) + 3I(1,1,0),\tag{Tao--Wu 16}\\
I(0,0,0)
&=0.011207724\ldots\tag{Tao--Wu 25}\\
I(1,0,0)
&=0.005621972\ldots\tag{Tao--Wu 26}\\
I(1,1,0)
&=0.001168842\ldots\tag{Tao--Wu 27}\\
c_3
&=0.0329390\ldots.\tag{Tao--Wu 30}
\end{align}
\end{subequations}
These values are not consistent:
substitution of their (25)--(27) into their (16) gives $c_3\approx 0.031580166$. 
The computation would be nearly correct if $I(0,0,0)\approx 0.011007724$ (perhaps miscopying ``two zeros'' as ``two zero'') and $I(1,1,0)\approx 0.00168842$ (mind the missing digit), but we still cannot entirely account for the discrepancy.
We do not attempt to correct this error in our tables in part because doing so makes the Tao--Wu estimates even worse.  For example, treating their (25)--(27) as correct leads to a downstream estimate of $P\approx 0.285410$.

As mentioned in the introduction, Winther~\cite{Winther:mse} reports a Monte Carlo simulation result of $0.28418(1)$ in an answer on the Mathematics Stack Exchange.  This simulation was faithful to the original problem description: a large number of points were sampled directly from the unit square $[0,1]^2$ and nearest neighbors were computed using an approach that is more efficient than brute-force search.  Winther's quoted error estimate of $10^{-5}$ is $3\sigma$ ($99.7\%$~confidence) of the standard error, including some consideration paid to the contribution of boundary effects.

We also performed a Monte Carlo simulation, but already in the limiting domain.  Assuming the origin is one of the nerds in the plane, we sample the other nerds according to a Poisson point process with constant rate.  Specifically, we sample the remaining points in increasing order of their distance from the origin.  This allows us to cut off the simulation after a finite number of points have been selected, since further points no longer affect who shoots the origin.  We sampled over $2\cdot 10^{11}$ configurations to estimate the values $c_n$ and the various $\smallint$~integrals.

For the purposes of estimating $P$, it is sometimes possible to sample even more efficiently.  Indeed, after sampling some number of points from the Poisson process in order of increasing radius, it is sometimes already known that the origin will be shot, but not yet known exactly how many times.  By cutting off such samples early, we were able to sample over $5\cdot 10^{12}$ configurations and estimate the value of $P$ more precisely.  For all of these Monte Carlo simulation estimates, we quote an error estimate of $1\sigma$ ($68\%$~confidence) in accordance with standard error reporting practice.

As an additional check on our numerical integration estimates, we also performed Monte Carlo \emph{integration}.  (The reasoning here is that Monte Carlo integration is relatively straightforward from a programming perspective, and compared to other numerical integration techniques, it is not as sensitive to the behavior of the integrand.)  We again quote an error estimate of $1\sigma$, combining estimates from different pieces by assuming they are independent and normally distributed.

Our numerical integration is consistent with the various Monte Carlo estimates.  The largest difference is between the numerical integration and the Monte Carlo simulation of $\twotwotwo$, which is $2.49$ standard deviations away.  Note that we report $38$ Monte Carlo estimates for various values, so one would reasonably expect some amount of spread in the estimates as a result.

\section*{Acknowledgments}

DGM was supported in part by NSF DMS 2220304.

\raggedright

\begin{landscape}

\begin{table}
\caption{\label{table.c_n's} Various estimates of $c_n$s}
\begin{center}
\begin{tabular}{lllll}
source & $c_2$ & $c_3$ & $c_4$ & $c_5$ \\ \hline
Tao--Wu~\cite{TaoW:87} 
& 0.3163335
& 0.0329390
& 0.0006575
& 0.0000010
\\
Finch~\cite{Finch:19}
& 0.316585
& 0.033056
& $-$
& $-$
\\
Monte Carlo integration
& 0.3165821(57)
& 0.0330571(25)
& 0.00065702(18)
& 0.0000002038025(94)
\\
Monte Carlo simulation
& 0.3165833(13)
& 0.03305604(40)
& 0.000657115(52)
& 0.00000020460(92)
\\
numerical integration
& 0.3165850647281(20)
& 0.0330563647606(88)
& 0.00065706696(46)
& 0.00000020380(48)
\end{tabular}
\end{center}
\end{table}

\begin{table}
\caption{\label{table.small ints} Various estimates of $\smallint$s and $I$s}
\begin{center}
\begin{tabular}{llllll}
source & \footnotesize{$\onefour=I(1,0)$} & \footnotesize{$\twothree=\frac{1}{2}I(0,0)$} & \footnotesize{$I(0,0)=2\twothree$} 
\\ \cline{1-4}
Tao--Wu~\cite{TaoW:87} 
&0.028880062
&0.129286084
&0.258572168
\\
Monte Carlo integration
& 0.0288809(20)
& 0.1294101(20)
& 0.2588203(40)
\\
Monte Carlo simulation
& 0.02888152(25)
& 0.12941012(56)
& 0.2588202(11)
\\
numerical integration
& 0.0288814929604(10)
& 0.1294110394036666(10)
& 0.2588220788073332(20)
\\
\\
source & \footnotesize{$\oneonethree=I(1,1,0)$} & \footnotesize{$\onetwotwo$} & \footnotesize{$\onetwothree$} & \footnotesize{$\twotwotwo=I(0,0,0)$} &
\footnotesize{$I(1,0,0)=\onetwotwo+2\onetwothree$}
\\ \hline
Tao--Wu~\cite{TaoW:87} 
&0.001168842
&--
&--
&0.011207724
&0.005621972
\\
Monte Carlo integration
& 0.00117493(32)
& 0.00448895(76)
& 0.000630599(19)
& 0.01228190(44)
& 0.00575015(76)
\\
Monte Carlo simulation
& 0.001174887(43)
& 0.004489011(83)
& 0.000630586(23)
& 0.01228098(22)
& 0.005750133(94)
\\
numerical integration
& 0.00117490461633(40)
& 0.00448886036115(40)
& 0.00063058779302(40)
& 0.0122815430701(40)
& 0.0057500359472(12)
\\
\\
source & \footnotesize{$\oneoneonetwo=I(1,1,1,0)$} & \footnotesize{$\oneonetwotwo=I(1,1,0,0)$} & \footnotesize{$\onetwoonetwo=I(1,0,1,0)$} & \footnotesize{$\onetwotwotwo=I(1,0,0,0)$} \\ \cline{1-5}
Monte Carlo integration
& 0.000057108(43)
& 0.0000640379(66)
& 0.000060504(13)
& 0.00001285586(73)
\\
Monte Carlo simulation
& 0.0000571294(77)
& 0.0000640431(81)
& 0.000060511(11)
& 0.0000128506(36)
\\
numerical integration
& 0.000057122200(80)
& 0.0000640437671(80)
& 0.000060491237(40)
& 0.0000128551537(80)
\\
\\
source & \footnotesize{$\oneoneoneonetwo=I(1,1,1,1,0)$} & \footnotesize{$\oneoneoneoneone=I(1,1,1,1,1)$} \\ \cline{1-3}
Monte Carlo integration
& 0.0000001904612(94)
& 0.00000000266825(17)
\\
Monte Carlo simulation
& 0.00000019142(89)
& 0.000000002635(47)
\\
numerical integration
& 0.00000019046(48)
& 0.00000000266817(16)
\end{tabular}
\end{center}
\end{table}

\end{landscape}

\end{document}